\documentclass[a4paper,11pt]{amsart}
  \usepackage{amsmath}
  \usepackage{url,graphicx}
  \usepackage{amssymb, color, pstricks-add, manfnt}
  \usepackage{amsmath, amsthm,  amsfonts}
  \usepackage{mathrsfs}
  \usepackage{ulem}

  \theoremstyle{plain}
  \newtheorem{Theorem}{Theorem}[section]
  \newtheorem{Lemma}{Lemma}[section]

  \newtheorem{Corollary}{Corollary}[section]

    \theoremstyle{remark}
  \newtheorem{remark}{Remark}

  \numberwithin{equation}{section}
  \numberwithin{figure}{section}
  \numberwithin{remark}{section}

\renewcommand{\baselinestretch}{1.00}
\usepackage{hyperref}
\usepackage{cite}

\begin{document}

\title{On the Dirichlet  problem for general augmented Hessian equations}

\author{Feida Jiang}
\address{College of Mathematics and Statistics, Nanjing University of Information Science and Technology, Nanjing 210044, P. R. China}
\email{jfd2001@163.com}

\author{Neil S. Trudinger}
\address{Centre for Mathematics and Its Applications, The Australian National University,
              Canberra ACT 0200, Australia; School of Mathematics and Applied Statistics, University of Wollongong, Wollongong, NSW 2522, Australia}
\email{Neil.Trudinger@anu.edu.au; neilt@uow.edu.au}

\thanks{Research supported by National Natural Science Foundation of China (No.11771214) and Australian Research Council (No.DP170100929).}


\date{\today}


\abstract{In this paper we apply various first and second derivative estimates and barrier constructions from our  treatment of oblique boundary value problems for  augmented Hessian equations,  to the case of Dirichlet boundary conditions. As a result we extend our previous results  on the Monge-Amp\`ere and $k$-Hessian cases to general classes of augmented Hessian equations in Euclidean space.}

\endabstract

\maketitle


\baselineskip=12.8pt
\parskip=3pt
\renewcommand{\baselinestretch}{1.38}

\section{Introduction}\label{Section 1}
\vskip10pt

In this paper we apply various first and second derivative estimates and barrier constructions from our treatment of oblique boundary value problems in \cite{JT-oblique-I, JT-oblique-II} to the classical Dirichlet problem for general classes of augmented Hessian equations,
thereby extending our previous results in \cite{JTY2013, JTY2014} on the Monge-Amp\`ere and $k$-Hessian cases.


We consider general augmented Hessian equations in the form,
\begin{equation}\label{1.1}
\mathcal{F}[u]:=F[D^2u-A(\cdot,u,Du)]=B(\cdot,u,Du), \quad {\rm in} \ \Omega,
\end{equation}
where the scalar function $F$ is defined on an open cone $\Gamma$ in $\mathbb{S}^n$, the linear space of  $n\times n$ real symmetric matrices,   $\Omega\subset \mathbb{R}^n$ is a bounded domain,
$A: \Omega \times \mathbb{R}\times \mathbb{R}^n\rightarrow \mathbb{S}^n$ is a symmetric matrix function and $B: \Omega \times \mathbb{R}\times \mathbb{R}^n\rightarrow \mathbb{R}$ is a scalar function. Our Dirichlet boundary conditions have the form
\begin{equation}\label{1.2}
\mathcal{G}[u]:=u-\varphi=0, \quad {\rm on} \ \partial\Omega,
\end{equation}
where $\varphi$ is a smooth function on $\partial\Omega$. As usual, $Du$ and $D^2u$ denote respectively the gradient vector and the Hessian matrix of the unknown function $u\in C^2(\Omega)$  and  we use $x, z, p$ and $r$ to denote  points  in $\Omega, \mathbb{R}, \mathbb{R}^n$ and $\mathbb{S}^n$, respectively. 

 Following \cite{JT-oblique-I, JT-oblique-II} we assume further that  cone $\Gamma$ in $\mathbb{S}^n$ is convex, with vertex at $0$, containing the positive cone $K^+$,  and that $F\in C^2(\Gamma)$ satisfies the basic conditions:
\begin{itemize}
\item[{\bf F1}:]
$F$ is strictly increasing in $\Gamma$, that is
\begin{equation}\label{F1 inequality}
F_r := F_{r_{ij}} = \left \{ \frac{\partial F}{\partial r_{ij}} \right \} >0, \ {\rm in} \ \Gamma.
\end{equation}

\vspace{0.2cm}

\item[{\bf F2}:]
$F$ is concave in $\Gamma$, that is
\begin{equation}\label{F2 inequality}
\sum_{i,j,k,l =1}^n \frac{\partial^2 F}{\partial r_{ij}\partial r_{kl}} \eta_{ij} \eta_{kl} \le 0, \ {\rm in} \ \Gamma,
\end{equation}
for all symmetric matrices $\{\eta_{ij}\}\in \mathbb{S}^n$.

\vspace{0.2cm}

\item[{\bf F3}:]
$F(\Gamma)=(a_0, \infty)$ for a constant $a_0\ge -\infty$ with
\begin{equation}
\sup_{r_0\in \partial\Gamma}\limsup_{r\rightarrow r_0} F(r) \le a_0.
\end{equation}

\end{itemize}
We say that an operator $\mathcal{F}$ satisfies the above properties if the corresponding function $F$ satisfies them. Note that we can take the constant $a_0$ in F3 to be $0$ or $-\infty$. We also say that $\mathcal F$ is orthogonally invariant if $F$ is given as a symmetric function $f$ of the eigenvalues $\lambda_1, \cdots,\lambda_n$ of the matrix $r$, with $\Gamma$ closed under orthogonal transformations. While it was not essential for our study of oblique boundary conditions in \cite{JT-oblique-I}, the orthogonal invariance property of $\mathcal{F}$ is critical for our study of the Dirichlet problem \eqref{1.1}-\eqref{1.2}; see the $A=0$ case in \cite{CNS-Hessian, Tru1995} for example. In the orthogonally invariant case, we use
    \begin{equation}
        \tilde \Gamma =\lambda(\Gamma) =\{\lambda\in \mathbb{R}^n \ | \ \lambda=(\lambda_1, \cdots, \lambda_n) \ {\rm are \ eigenvalues\  of\ some} \ r\in \Gamma\}
    \end{equation}
to denote the corresponding cone to $\Gamma$ in $\mathbb{R}^n$. For convenience of later usage, we define for $k=1,\cdots, n$, the $k$ cone
    \begin{equation}\label{Garding's cone}
        \Gamma_k =\{r\in \mathbb{S}^n|\ S_j[r]>0, \ \forall j=1,\cdots, k\},
    \end{equation}
where $S_k$ denotes the $k$-th order elementary symmetric function defined by
    \begin{equation}
        S_k[r]:=S_k(\lambda(r)) = \sum_{i_1<\cdots <i_k}\lambda_{i_1} \cdots \lambda_{i_k}, \quad k=1,\cdots, n.
    \end{equation}

We call $M[u]:=D^2u-A(\cdot, u, Du)$ the augmented Hessian matrix, which is the standard Hessian matrix adjusted by subtraction of a lower order symmetric matrix function. A $C^2$ function $u$ is admissible in $\Omega$ ($\bar \Omega$), if
    \begin{equation}\label{admissible}
    M[u]\in \Gamma, \ {\rm in}\ \Omega, \ (\bar\Omega),
    \end{equation}
so that the operator $\mathcal{F}$ satisfying F1 is elliptic with respect to $u$ in $\Omega$ ($\bar \Omega$) when \eqref{admissible} holds. If an admissible function $u$ satisfies equation \eqref{1.1}, we call $u$ an admissible solution of equation \eqref{1.1}. Since $\mathcal{F}$ satisfies F3, the requirement $B>a_0$ in $\Omega$ ($\bar \Omega$) is necessary for an admissible solution of equation \eqref{1.1}. A function $\underline u\in C^2(\bar \Omega)$ is said to be admissible with respect to $u$ in $\Omega$ ($\bar \Omega$), if
    \begin{equation}
    M_u[\underline u] := D^2 \underline u - A(\cdot, u, D\underline u) \in \Gamma,\ {\rm in} \ \Omega, \ (\bar \Omega).    
    \end{equation}
Clearly if $A$ is independent of $z$, then $M_u[\underline u] = M[\underline u]$ so that $\underline u$ is admissible with respect to $u$ if and only if $\underline u$ is admissible. While if $A$ is non-decreasing in $z$, (non-increasing in $z$), then $M_u[\underline u] \ge M[\underline u]$ and $\underline u$ is admissible with respect to $u$, if $\underline u$ is admissible and $\underline u \ge u$, ($\le u$). If a function $\underline u$ ($\bar u$) satisfies
    \begin{equation}\label{sub sol}
    F(M_u[\underline u])\ge B(\cdot,u, D\underline u), \ (F(M_u[\bar u])\le B(\cdot,u, D\bar u)), 
    \end{equation}
at points in $\Omega$, we call $\underline u$ ($\bar u$) a subsolution (supersolution) of equation \eqref{1.1}. Moreover, we call $\underline u$ ($\bar u$) an admissible subsolution (supersolution) of equation \eqref{1.1} if $\underline u$ ($\bar u$) is admissible with respect to $u$.

The matrix $A$ is called regular (strictly regular), if
    \begin{equation}\label{regular}
    \sum_{i,j,k,l}^n A_{ij}^{kl}(x,z,p)\xi_i\xi_j\eta_k\eta_l \ge 0, \ (>0),
    \end{equation}
for all $(x,z,p)\in \Omega\times \mathbb{R} \times \mathbb{R}^n$, $\xi, \eta\in \mathbb{R}^n$ and $\xi\cdot \eta=0$, where $A_{ij}^{kl}=D^2_{p_kp_l}A_{ij}$. The regular condition \eqref{regular} was first introduced for the interior regularity in the context of optimal transportation in \cite{MTW2005} in its strict form, and subsequently used for the global regularity in \cite{TruWang2009} in its weak form. If \eqref{regular} holds without the restriction $\xi\cdot \eta=0$, the matrix $A$ is called regular without orthogonality. Note that the case when $A=A(x,z)$, and in particular the basic Hessian case $A\equiv 0$, satisfies the regular condition without orthogonality.

We now begin to formulate the main theorems of this paper. \begin{Theorem}\label{Th1.1}
Let $u\in C^4(\Omega)\cap C^2(\bar \Omega)$ be an admissible solution of Dirichlet problem \eqref{1.1}-\eqref{1.2}, where $\mathcal{F}$ is orthogonally invariant and satisfies F1-F3 in $\Gamma\subset \Gamma_1$, $A\in C^2(\bar \Omega\times\mathbb{R}\times \mathbb{R}^n)$ is regular in $\bar \Omega$, $B>a_0, \in C^2(\bar \Omega\times \mathbb{R}\times\mathbb{R}^n)$ is convex with respect to $p$. Assume there exists an admissible subsolution $\underline u\in C^2(\bar \Omega)$ satisfying \eqref{sub sol} in $\Omega$ and $\underline u=\varphi$ on $\partial\Omega$, with $\varphi\in C^4(\partial\Omega)$, $\partial\Omega\in C^4$. Then we have the estimate
    \begin{equation}\label{full 2nd estimate}
        \sup_{\Omega}|D^2u| \le C,
    \end{equation}
where the constant $C$ depends on $n, A, B, \Omega, \varphi, \underline u$, and $|u|_{1;\Omega}$.
\end{Theorem}

Note that for the second derivative estimate in Theorem \ref{Th1.1}, $\mathcal{F}$ is only assumed to be orthogonally invariant and satisfy the basic conditions F1, F2 and F3, where $a_0$ can be either finite or infinite. 
Since the global second derivative estimate of the form $\sup\limits_{\Omega}|D^2u|\le C(1+ \sup\limits_{\partial\Omega}|D^2u|)$ is already established in \cite{JT-oblique-II}, in order to prove the estimate \eqref{full 2nd estimate}, it is enough to obtain the boundary estimate for $\sup\limits_{\partial\Omega}|D^2u|$, which relies on the construction of the appropriate barrier functions. Such a technical barrier construction will be discussed in Lemma \ref{Lemma 2.1}, as well as its strengthened version in Lemma \ref{Lemma 2.2}.

For gradient estimates, there are a range of conditions on $F$, $A$ and $B$. In  particular we recall a further condition for $F$  from  \cite{JT-oblique-I}, which along with F1-F3 is satisfied by our examples in Section 4 of  \cite{JT-oblique-I}, 
\begin{itemize}
\item[{\bf F7}:]
For a given constant $a>a_0$, there exists constants $\delta_0, \delta_1>0$ such that 
$$F_{r_{ij}}\xi_i\xi_j \ge \delta_0 + \delta_1 \mathscr{T},$$ 
if $a\le F(r)$ and $\xi$ is a unit eigenvector of $r$ corresponding to a negative eigenvalue, where $\mathscr{T}={\rm trace}(F_r)$.
\end{itemize} 
To apply F7, apart from orthogonal invariance, we also need (almost quadratic) structure conditions on $A$ and $B$, which we write here in a fairly general form:
\begin{equation}\label{structure condition 1}
A =o(|p|^2)I, \quad\quad p\cdot D_pA \le O(|p|^2) I, \quad\quad  p\cdot D_pB \le O(|p|^2),
\end {equation}
\begin{equation}\label{structure condition 2}
p\cdot D_xA +|p|^2 D_zA\ge o(|p|^4)I, \quad\quad p\cdot D_xB +|p|^2 D_zB\ge o(|p|^4), 
\end{equation} 
as $|p|\rightarrow \infty$, uniformly for $x\in \Omega$, $|z|\le M$ for any $M> 0$, where $I$ denotes the $n\times n$ identity matrix. A global gradient estimate  then follows from our proof of case (ii) of Theorem 1.3 in Section 3 of \cite{JT-oblique-I}, while for a local gradient estimate we need to strengthen the last two inequalities in \eqref {structure condition 1}:
\begin{equation}\label{structure condition 3}
D_pA, D_pB = O(|p|).
\end{equation}
Note that in the special case when $\Gamma=K^+$, we only need the  one-sided quadratic structure $A\ge O(|p|^2)I$, as $|p|\rightarrow \infty$, uniformly for $x\in \Omega$, $|z|\le M$, for any $M> 0$, \cite{JTY2013}, while from \cite{JT-oblique-I}, if $\Gamma=\Gamma_k$ for $n/2 <k<n$, we can weaken ``$o$'' to ``$O$'' in \eqref{structure condition 2}, (at least when $a_0$ is finite), with \eqref{structure condition 1} replaced by \eqref{structure condition 3}, that is a quadratic structure is sufficient. By a slight modification of our arguments in Section 3 of \cite{JT-oblique-I},  we can use F2 instead of F7 in the global gradient bound, under some slight strengthening of our conditions on $F,A$ and $B$ which, for example, would still embrace the basic examples of functions $F$ which are positive homogeneous of degree one and involve replacing ``$O$'' by ``$o$'' throughout \eqref{structure condition 1}. These alternative conditions are also discussed in the case of oblique boundary conditions in Section 3 of \cite{JT-oblique-III}.
 
Further conditions for gradient bounds for strictly regular $A$ are given in \cite{JT-oblique-I}. These global gradient estimates  reduce the full gradient bound to the gradient bound on the boundary, which is readily deduced under our assumptions; see Section \ref{Section 3}.

The maximum modulus estimate for solution $u$ of the Dirichlet problem \eqref{1.1}-\eqref{1.2} is guaranteed by assuming the existence of  an admissible subsolution and a supersolution of the problem. Since we assume the admissible subsolution $\underline u\in C^2(\bar \Omega)$ satisfies $\underline u=\varphi$ on $\partial\Omega$, we already have the lower solution bound $u\ge \underline u$ in $\bar \Omega$ by the comparison principle. For the upper solution bound, we can assume $-A(x,z, 0)\notin \Gamma$ for all $x\in \Omega$ and $z\in \mathbb{R}$, which implies large constant functions are supersolutions. More generally we can assume that there exists a bounded viscosity supersolution $\bar u$ as in \cite{JTY2013}, so that $u\le \bar u$ on $\bar \Omega$.

We now formulate the following existence theorem for classical admissible solutions, where $A$ and $B$ are independent of $z$.
\begin{Theorem}\label{Th1.2}
Assume that $\mathcal{F}$ is orthogonally invariant and satisfies F1-F3, F7 in $\Gamma\subset \Gamma_1$, $\Omega$ is a bounded domain in $\mathbb{R}^n$ with $\partial\Omega\in C^4$, $A\in C^2(\bar \Omega\times \mathbb{R}^n)$ is regular in $\bar \Omega$, $B>a_0, \in C^2(\bar \Omega\times \mathbb{R}^n)$ is convex with respect to $p$. Assume there exist a bounded viscosity supersolution $\bar u$ and a subsolution $\underline u\in C^2(\bar \Omega)$ satisfying $\underline u = \varphi$ on $\partial\Omega$ with $\varphi\in C^4(\partial\Omega)$. Assume also either \eqref{structure condition 1} and \eqref{structure condition 2} hold or $\Gamma=K^+$, and $A(x,p)\ge O(|p|^2)I$ as $|p|\rightarrow \infty$, uniformly for $x\in \Omega$. 
 Then there exists a unique admissible solution $u\in C^3(\bar \Omega)$ of the Dirichlet problem \eqref{1.1}-\eqref{1.2}. Moreover, if $\Gamma=\Gamma_k$ for $n/2<k<n$, the conclusion still holds by replacing  \eqref{structure condition 1} by \eqref{structure condition 3} and weakening``$o$'' to ``$O$'' in \eqref{structure condition 2}.
\end{Theorem}

Corresponding to our remarks above, we can relax condition F7, at least for finite $a_0$, in the above hypotheses for general $\Gamma$, provided ``$O$'' is strengthened to ``$o$'' in \eqref{structure condition 1}, see Corollary \ref{Cor 3.1}. Since $A$ and $B$ are independent of $z$, it is convenient to call $\underline u$ a subsolution as usual in Theorem \ref{Th1.2}, rather than an admissible subsolution.

Historically, the Dirichlet problem of the standard Hessian equations for general operators has been studied extensively in \cite{CNS-Hessian, Guan1994, Tru1995} and our conditions F1-F3 correspond to the basic conditions in these works. Second derivative estimates and the existence results are established under an associated uniform convexity of the domain or the existence of an admissible subsolution. Both the domain convexity and the subsolution are used to construct barrier functions, which are then used in the derivation of boundary second derivative estimates. For the Dirichlet problem of the augmented Hessian equations, we have treated the Monge-Amp\`ere case in \cite{JTY2013} and the $k$-Hessian case in \cite{JTY2014}, for regular matrices $A$, under the existence of a subsolution, which is also used to obtain the global second  derivative bounds. There are also recent studies of the Dirichlet problem \eqref{1.1}-\eqref{1.2} on  Riemannian manifolds, under more restrictive conditions on the matrix function $A$, \cite{Guan2014, GJ2015, GJ2016}, where the existence of a subsolution is also critical for such bounds. These stem from the basic Hessian case in \cite{Guan2014}, where such a technique is developed independently of our discovery through the Monge-Ampere case in \cite{JTY2013}. We also remark that our treatment here will also extend to the more general Riemannian manifold case and as well the condition F3 can be weakened as for example in \cite{ITW2004}; (see also \cite{JT-oblique-III}).



The essential ingredients in this paper are already in our papers \cite{JT-oblique-I, JT-oblique-II}. These are the global second derivative estimates in Section 3 of \cite{JT-oblique-II} and the global gradient estimates in Section 3 of \cite{JT-oblique-I}, in particular Remark 3.1. In Section \ref{Section 2} of this paper, we obtain the second derivative estimates on the boundary following the methods already established \cite{CNS-Hessian, Guan1994, Tru1995}   and thus complete the proof of Theorem \ref{Th1.1}. A strengthened technical barrier construction, already invoked for the basic Hessian case in \cite{Guan2014}, is also discussed, which provides an alternative approach to the estimates of the mixed tangential-normal derivatives and pure normal derivatives on the boundary. In Section \ref{Section 3}, we consider alternative gradient estimate hypotheses  and in particular derive the gradient estimate, with F7 replaced by F2,  by modification of our argument in case (ii) of Theorem 1.3 in \cite{JT-oblique-I}. Finally, we prove the existence of classical admissible solutions in Theorem \ref{Th1.2} by the method of continuity.

\section{Boundary estimates for second derivatives}\label{Section 2}
\vskip10pt

In this section, we shall make full use of the admissible subsolution $\underline u\in C^2(\bar \Omega)$ of equation \eqref{1.1} to establish the second derivative estimate $|D^2u|\le C$ on $\partial\Omega$. Together with the global second derivative bound in terms of its boundary bound in Theorem 3.1 in \cite{JT-oblique-II}, we can get full second derivative estimate \eqref{full 2nd estimate} in Theorem \ref{Th1.1} based on the boundary estimate in this section. We also discuss a new barrier construction, which provides a more direct approach in both the mixed tangential-normal derivative estimate and the pure normal derivative estimate on the boundary.

By a standard perturbation argument, we can make a non-strict admissible subsolution $\underline u\in C^2(\bar \Omega)$ of equation \eqref{1.1} to be a strict admissible $C^2$ subsolution of equation \eqref{1.1}. Similarly, for the admissible subsolution $\underline u$, if we restrict it in a neighbourhood of $\partial\Omega$, we can modify it to be a strict admissible subsolution satisfying the same boundary condition. It is also readily checked that the form of the equation \eqref{1.1} and the regularity condition \eqref{regular} can be preserved under translation and rotation of coordinates.

We now proceed to the boundary estimates. For any given point $x_0\in \partial\Omega$, by a translation and a rotation of the coordinates, we may take $x_0$ as the origin and $x_n$ axis to be the inner normal of $\partial\Omega$ at the origin. Near the origin, $\partial\Omega$ can be represented as a graph
$$x_n=\rho (x'),$$
such that $D^\prime \rho (0)=0$, where $D^\prime =(D_1,\cdots, D_{n-1})$ and $x'=(x_1, \cdots, x_{n-1})$. By tangentially differentiating \eqref{1.2} twice, we have
\begin{equation}\label{twice diff boundary}
D_{\alpha\beta}(u-\varphi) (0) = - D_n (u-\varphi)(0) \rho_{\alpha\beta}(0), \quad \alpha, \beta = 1,\cdots, n-1,
\end{equation}
which leads to the pure tangential estimate,
\begin{equation}\label{pure tangential}
|D_{\alpha\beta}u(0)|\le C, \quad \alpha, \beta=1,\cdots, n-1,
\end{equation}
where the constant $C$ depends on $\Omega$, $\varphi$ and $|u|_{1;\Omega}$.

We then estimate the mixed tangential-normal derivatives $|D_{\alpha n}u(0)|$ for $\alpha=1,\cdots, n-1$, by using a barrier argument. For this estimate, we consider the following operator
\begin{equation}\label{T operator}
T_\alpha :=\partial_\alpha + \sum_{\beta<n} \rho_{\alpha\beta}(0)(x_\beta \partial_n - x_n \partial_\beta), \quad {\rm for \ fixed} \ \alpha<n.
\end{equation}
By calculations, we have for $\alpha<n$,
    \begin{equation}\label{LTu}
        \begin{array}{rl}
            \mathcal{L}(T_\alpha u)
             = \!\!\! & \displaystyle \mathcal{L} u_\alpha +\sum_{\beta <n}\rho_{\alpha\beta}(0)(x_\beta \mathcal{L}u_n-x_n \mathcal{L}u_\beta)\\
             \!\!\! & \displaystyle + \sum_{\beta
            <n}\rho_{\alpha\beta}(0) [2(F^{\beta j}u_{nj}-F^{nj}u_{\beta j})\\
             \!\!\! & \displaystyle -F^{ij}(A_{ij}^\beta u_n-A_{ij}^n u_\beta)-(u_n D_{p_{\beta}}B-u_\beta D_{p_{n}}B) ],
        \end{array}
    \end{equation}
where $\mathcal{L}$ is the linearized operator defined by
\begin{equation}\label{mathcal L def}
\mathcal{L}:= F^{ij}D_{ij} - (F^{ij}A_{ij}^k-D_{p_k}B)D_k,
\end{equation}
with
$$F^{ij}:=\frac{\partial F}{\partial w_{ij}}, \quad {\rm and} \ \  A_{ij}^k:=D_{p_k}A_{ij},$$
where $w_{ij}:=u_{ij}-A_{ij}$.
By differentiation equation \eqref{1.1} with respect to $x_k$, we have
        \begin{equation}\label{LDku}
        \mathcal {L} u_k =F^{ij}(A_{ij}^k+u_k D_zA_{ij} )+ (D_{x_k}B+u_k D_zB),\ \ k=1,
        \cdots, n.
    \end{equation}
Since $\mathcal{F}$ is orthogonal invariant, we can derive, for $\alpha<n$,
    \begin{equation}
    \sum_{\beta <n}\rho_{\alpha\beta}(0) (F^{\beta j}w_{nj}-F^{nj}w_{\beta j})=0,
    \end{equation}
so that
    \begin{equation}\label{infinitesimal}
    \sum_{\beta <n}\rho_{\alpha\beta}(0) (F^{\beta j}u_{nj}-F^{nj}u_{\beta j})=\sum_{\beta <n}\rho_{\alpha\beta}(0) (F^{\beta j}A_{nj}-F^{nj}A_{\beta j}).
    \end{equation}
From \eqref{LTu}, \eqref{LDku} and \eqref{infinitesimal}, we obtain
\begin{equation}\label{standard inequality}
|\mathcal{L}(T_\alpha (u-\varphi))| \le C (1+\mathscr{T}), \quad {\rm in} \ \Omega,
\end{equation}
for $\alpha<n$, where the constant $C$ depends on $\Omega, A, B, \varphi$ and $|u|_{1;\Omega}$. For $\alpha<n$, we also have
\begin{equation}
|T_\alpha (u-\varphi)| \le C|x|^2, \quad {\rm on} \ \partial\Omega.
\end{equation}

We are now in a position to employ an appropriate barrier function. We present the following lemma without proof, which is a restatement of the general barrier construction in Lemma 2.1(ii) in \cite{JT-oblique-II}.
\begin{Lemma}\label{Lemma 2.1}
Let $u\in C^2(\bar \Omega)$ be an admissible solution of equation \eqref{1.1}, $\underline u\in C^2(\bar \Omega)$ be an admissible strict subsolution of equation \eqref{1.1} satisfying 
\begin{equation}\label{strict sb}
F(M_u[\underline u]) > B(\cdot, u, D\underline u), \quad {\rm in} \ \Omega.
\end{equation}
Assume $\mathcal{F}$ satisfies F1-F3, $A\in C^2(\bar \Omega\times\mathbb{R}\times \mathbb{R}^n)$ is regular in $\bar \Omega$, $B>a_0, \in C^2(\bar \Omega\times\mathbb{R}\times \mathbb{R}^n)$ is convex in $p$. Then there exist positive constants $K$ and $\epsilon_1$, depending on $\Omega, A, B, \underline u$ and $|u|_{1;\Omega}$, such that
\begin{equation}\label{barrier in Lemma 2.1}
\mathcal{L} \left [e^{K(\underline u-u)} \right ]\ge \epsilon_1 (1+\mathscr{T}), \quad {\rm in}\ \Omega,
\end{equation}
where $\mathcal{L}$ is the linearized operator defined in \eqref{mathcal L def}, and $\mathscr{T}={\rm trace}(F_r)$.
\end{Lemma}

By applying Lemma \ref{Lemma 2.1} to our strict subsolution $\underline u$ satisfying \eqref{strict sb} in the neighbourhood of the boundary, we then have
\begin{equation}\label{2.27}
\mathcal{L} \eta \ge \epsilon_1 (1+\mathscr{T}), \quad {\rm in}\ \Omega_\rho,
\end{equation}
with $\eta :=\exp[K(\underline u-u)]$, $\underline u=u=\varphi$ on $\partial\Omega$, and $\Omega_\rho:=\{x\in \Omega|
\ d(x)<\rho \}$. Then the function $\hat \eta =1 -\eta$ satisfies
\begin{equation}
\begin{array}{cl}
\mathcal{L}\hat \eta \le -\epsilon_1 (1+\mathscr{T}), & \quad {\rm in} \ \Omega_{\rho},\\
           \hat \eta  = 0, & \quad {\rm on} \ \partial\Omega, \\
           \hat \eta  > 0, & \quad {\rm on} \ \partial\Omega_\rho \cap \Omega.
\end{array}
\end{equation}
Letting
\begin{equation}\label{final barrier}
\tilde \eta = a \hat \eta + b|x|^2,
\end{equation}
with positive constants $a \gg b \gg 1$, we then have for $\alpha<n$,
\begin{equation}
\begin{array}{rll}
\mathcal{L}\tilde \eta  + |\mathcal{L} (T_\alpha(u-\varphi))| \!\!&\!\!\displaystyle \le 0, & \quad {\rm in} \ \Omega_{\rho},\\
           \tilde \eta + |T_\alpha(u-\varphi)|\!\!&\!\!\displaystyle \ge 0, & \quad {\rm on} \ \partial\Omega_\rho.
\end{array}
\end{equation}
By the maximum principle, we derive the mixed tangential-normal derivative estimate
\begin{equation}\label{tangential-normal}
|D_{\alpha n}u(0)| \le C, \quad {\rm for} \ \alpha=1, \cdots, n-1,
\end{equation}
where the constant $C$ depends on $\Omega, A, B, \varphi, \underline u$ and $|u|_{1;\Omega}$.


Up to now, from \eqref{pure tangential} and \eqref{tangential-normal}, the following estimates on the boundary are already under control,
\begin{equation}\label{M2 prime estimate}
|D_{ij}u(x)| \le M^\prime_2, \quad{\rm for} \  i+j <2n, \ x \in \partial\Omega,
\end{equation}
where the constant $M^\prime_2$ depends on $\Omega, A, B, \varphi, \underline u$ and $|u|_{1;\Omega}$. In \eqref{M2 prime estimate}, the coordinate systerm is chosen so that the positive axis is directed along the inner normal at the point $x\in \partial\Omega$.

The remaining estimate is the pure normal second order derivative estimate on the boundary. For this estimation, we shall use the idea in \cite{Tru1995}. Since $\Gamma\subset \Gamma_1$, the lower bound for $u_{nn}$ is direct from ${\rm trace}(M[u])>0$ and \eqref{M2 prime estimate}. We need to derive an upper bound for $u_{nn}$ on $\partial\Omega$.

For any boundary point $x\in \partial\Omega$, fixing a principal coordinate system at the point $x$ and a corresponding neighbourhood $\mathcal{N}$ of $x$ with $\gamma_n\in [-1, -1/2)$ in $\mathcal{N}\cap \partial\Omega$, we let $\xi^{(1)}, \cdots, \xi^{(n-1)}$ be an orthogonal vector field on $\mathcal{N}\cap\partial\Omega$, which is tangential to $\mathcal{N}\cap\partial\Omega$, namely $\xi^{(j)}\cdot \gamma =0$ for $j=1, \cdots, n-1$, where $\gamma$ is the unit outer normal vector field on $\partial\Omega$. Note that the vector field $\xi^{(1)}, \cdots, \xi^{(n-1)}$ agrees with the coordinate system at $x$, namely $\xi_\alpha^{(\beta)}(x)=\delta_{\alpha\beta}$, $\alpha, \beta=1,\cdots, n-1$. We introduce the following notations
$$
\nabla_\alpha u=\xi_m^{(\alpha)}D_m u, \quad \nabla_{\alpha\beta}u = \xi_m^{(\alpha)}\xi_l^{(\beta)}D_{ml}u, \quad \mathfrak{C}_{\alpha\beta} =\xi_m^{(\alpha)}\xi_l^{(\beta)}D_m\gamma_l,\quad 1\le \alpha, \beta \le n-1,
$$
$$
\nabla u =(\nabla_1 u, \cdots, \nabla_{n-1}u), \quad D_\gamma u = \gamma_m D_m u,
$$
$$
\mathcal{A}_{\alpha\beta}(\cdot, u, \nabla u, -D_\gamma u)= \xi_m^{(\alpha)}\xi_l^{(\beta)}A_{ml}(\cdot, u, \nabla u, -D_\gamma u),\quad  1\le \alpha, \beta \le n-1,
$$
$$
\omega_{\alpha\beta} =\xi_m^{(\alpha)}\xi_l^{(\beta)}w_{ml}=\xi_m^{(\alpha)}\xi_l^{(\beta)}(D_{ml}u-A_{ml}), \quad 1\le \alpha, \beta \le n-1,
$$
$$
\omega_{-\gamma \alpha}=-\gamma_m \xi_l^{(\alpha)}w_{ml}, \quad \omega_{-\gamma,\alpha} = \{\omega_{-\gamma 1}, \cdots, \omega_{-\gamma (n-1)}\}^T, \quad 1\le \alpha, \beta \le n-1,
$$
and
$$
\nabla^2 u = \{\nabla_{\alpha\beta}u\}_{1\le \alpha,\beta \le n-1}, \quad \mathfrak{C}=\{\mathfrak{C}_{\alpha\beta}\}_{1\le \alpha,\beta \le n-1}, \quad \mathcal{A} =\{\mathcal{A}_{\alpha\beta}\}_{1\le \alpha,\beta \le n-1},
$$
Since  $u=\varphi$ on $\partial\Omega$, we have for $x\in \partial\Omega$,
$$\mathcal{A}_{\alpha\beta}(x, u(x), \nabla u(x), - D_\gamma u(x))=A_{\alpha\beta}(x, \varphi(x), D^\prime \varphi (x), D_n u(x)),$$
for $1\le \alpha, \beta \le n-1$, and 
$$B(x, u(x), \nabla u(x), - D_\gamma u(x))=B(x, \varphi(x), D^\prime \varphi (x), D_n u(x)),$$ 
where $D^\prime = (D_1, \cdots, D_{n-1})$.
The augmented Hessian matrix under the orthogonal vector field $\xi^{(1)}, \cdots, \xi^{(n-1)}$ and $\gamma$, can be written as
\begin{equation}\label{M'[u]}
M[u] := \left (
\begin{array}{cc}
M^\prime [u],          & \omega_{-\gamma, \alpha} \\
\omega_{-\gamma, \alpha}^T, & \omega_{\bar\gamma \bar\gamma}
\end{array}
\right ),
\end{equation}
where $M^\prime[u] := \nabla^2 u-\mathcal{A}$, and $\bar \gamma=-\gamma$.
From the boundary condition $u=\varphi$ on $\partial \Omega$, we have
    \begin{equation}\label{nabla2u-varphi}
        \nabla^2(u-\varphi)=D_\gamma(u-\varphi)\mathfrak{C}
    \end{equation}
on $\partial\Omega$, which agrees with \eqref{twice diff boundary} at $x_0$. We then have, on the boundary $\partial \Omega$,
    \begin{equation}\label{Mprimeu}
        M^\prime [u]
        =\{ D_\gamma(u-\varphi)\mathfrak{C}_{\alpha\beta}+\nabla_{\alpha\beta}\varphi-\mathcal{A}_{\alpha\beta}(x,\varphi, D^\prime \varphi,-D_\gamma u)\}_{\alpha,\beta=1,\cdots,n-1}.
    \end{equation}

For a sufficiently large constant $R$ satisfying $(\lambda^\prime(M^\prime[u]), R)\in \tilde \Gamma$, we define
\begin{equation}\label{f infty}
f_R (\lambda^\prime(M^\prime[u])) := F(\tilde M[u]+ R \gamma\otimes \gamma),
\end{equation}
where $\lambda^\prime=\lambda_x^\prime=(\lambda_1,\cdots, \lambda_{n-1})$ are the eigenvalues of $M^\prime [u]:=\{\omega_{\alpha\beta}\}_{\alpha,\beta=1,\cdots, n-1}$, and 
\begin{equation}
\tilde M[u] := \left (
\begin{array}{cc}
M^\prime [u],          & 0 \\
0, & 0
\end{array}
\right ),
\end{equation}
We now fix a point $x_0\in \partial\Omega$, where the function $g$ defined by
\begin{equation}\label{g function}
g(x):= f_R(\lambda^\prime(M^\prime [u])) - B[u]
\end{equation}
is minimized over $\partial\Omega$, where $M^\prime [u]:=\{\omega_{\alpha\beta}\}_{\alpha,\beta=1,\cdots, n-1}$, $B[u]:=B(\cdot, u, Du)$, $R$ is a sufficiently large constant satisfying $(\lambda^\prime(M^\prime[u]), R)\in \tilde \Gamma$. In order to derive an upper bound for $u_{nn}$ on $\partial\Omega$, we aim to get a positive lower bound for the function $g(x)$ in \eqref{g function} on $\partial\Omega$.

We assume that the function $h$ defined by
\begin{equation}\label{h function}
h(x):= f_R(\lambda^\prime(M^\prime_u[\underline u])) - B_u[\underline u]
\end{equation}
is minimized over $\partial\Omega$ at a point $y\in \partial\Omega$, where $M^\prime_u[\underline u] := \{D_{\alpha\beta} \underline u - A_{\alpha\beta}(\cdot, u, D\underline u)\}_{\alpha,\beta=1,\cdots, n-1}$, and $B_u[\underline u]=B(\cdot, u, D\underline u)$. It is obvious that $h(x_0)\ge h(y)>0$. Since $\underline u=\varphi$ on $\partial\Omega$, similarly to \eqref{nabla2u-varphi} and \eqref{Mprimeu}, we also have
\begin{equation}\label{nabla2underu-varphi}
\nabla^2(\underline u-\varphi)=D_\gamma(\underline u-\varphi)\mathfrak{C}
\end{equation}
on $\partial\Omega$, and
\begin{equation}\label{Mprimeunderu}
\begin{array}{ll}
M_u^\prime [\underline u] \!\!&\!\! := \nabla^2 \underline u-\mathcal{A}(x,u,D\underline u)\\
\!\!&\!\! =\{ D_\gamma(\underline u-\varphi)\mathfrak{C}_{\alpha\beta}+\nabla_{\alpha\beta}\varphi-\mathcal{A}_{\alpha\beta}(x,\varphi, D^\prime \varphi,-D_\gamma \underline u)\}_{\alpha,\beta=1,\cdots,n-1}.
\end{array}
\end{equation}

Similar to \eqref{M'[u]}, we can write a symmetric matrix $r\in \mathbb{S}^n$ as 
    \begin{equation}
	r := \left (
	\begin{array}{cc}
	r^\prime,          & r_{\alpha,n} \\
	r_{\alpha,n}^T, & r_{nn}
	\end{array}
	\right ),
    \end{equation}
where $r^\prime\in \mathbb{S}^{n-1}$, $r_{\alpha, n}= \{r_{1n}, \cdots, r_{(n-1)n}\}^T$. Let $\Gamma^\prime:=\{r^\prime \in \mathbb{S}^{n-1} |\  r\in \Gamma\}$ be the projection cone of the cone $\Gamma\subset \mathbb{S}^{n}$ onto $\mathbb{S}^{n-1}$, and $\tilde \Gamma^\prime$ be the corresponding cone to $\Gamma^\prime$ in $\mathbb{R}^{n-1}$. For any matrix $r^\prime\in \Gamma^\prime$, with eigenvalues $\lambda_1, \cdots, \lambda_{n-1}$, let us now define
    \begin{equation}\label{G function}
    G(r^\prime):= f_R (\lambda_1, \cdots, \lambda_{n-1}),
    \end{equation}
and
    \begin{equation}
    G^{\alpha\beta} = \frac{\partial G}{\partial r^\prime_{\alpha\beta}}, \quad G_{x_0}^{\alpha\beta} = G^{\alpha\beta}(M^\prime[u](x_0)),
    \end{equation}
for $1\le \alpha, \beta \le n-1$.
Since the function $f_R$ is non-decreasing and concave in the cone $\tilde \Gamma^\prime$ from F1 and F2, then the function $G$ is non-decreasing and concave in the cone $\Gamma^\prime$, see \cite{CNS-Hessian, Tru1995}. From \eqref{g function}, we have
    \begin{equation}
    g(x) \ge g(x_0), \quad {\rm for} \ x\in \partial\Omega,
    \end{equation}
namely,
    \begin{equation}\label{f infty - B}
    f_R(\lambda^\prime(M[u](x))) - B[u](x) \ge f_R(\lambda^\prime(M[u](x_0))) - B[u](x_0),
    \end{equation}
on $\partial\Omega$.
From \eqref{G function} and \eqref{f infty - B}, we have, on $\partial\Omega$,
    \begin{equation}\label{G - B}
    G(M^\prime[u](x)) - B[u](x) \ge G(M^\prime[u](x_0)) - B[u](x_0).
    \end{equation}
From the concavity of $G$, we then have, on $\partial\Omega$,
    \begin{equation}\label{starting point 1}
    G^{\alpha\beta}_{x_0} (\omega_{\alpha \beta}(x) - \omega_{\alpha \beta}(x_0)) - B[u](x) + B[u](x_0) \ge 0.
    \end{equation}

We consider the two possible cases:

{\it Case 1.} $g(x_0) \ge h(x_0)/2$. Since this inequality can provide a positive lower bound for $g(x)$, we are done.

{\it Case 2.} $g(x_0)< h(x_0)/2$. By successively using \eqref{g function}, \eqref{h function}, \eqref{G function}, the concavity of $G$, \eqref{Mprimeu} and \eqref{Mprimeunderu}, we have
    \begin{equation}\label{g-h 1}
    \begin{array}{rl}
     \!\!&\!\!\displaystyle  g(x_0)-h(x_0) \\
      = \!\!&\!\!\displaystyle  \left(G(M^\prime [u](x_0))-B[u](x_0)\right) - \left(G(M^\prime_u [\underline u](x_0))-B_u[\underline u](x_0)\right) \\
      \ge \!\!&\!\!\displaystyle  G^{\alpha\beta}_{x_0} \left( \{M^\prime [u](x_0)\}_{\alpha\beta} - \{M^\prime_u [\underline u](x_0)\}_{\alpha\beta}\right ) - B[u](x_0) + B_u[\underline u](x_0)\\
      = \!\!&\!\!\displaystyle  -G^{\alpha\beta}_{x_0} \left [D_n(u-\underline u)(x_0)D_\alpha\gamma_\beta(x_0)+A_{\alpha\beta}(x_0, \varphi(x_0), D'\varphi(x_0), D_nu(x_0))\right.\\
        \!\!&\!\!\displaystyle  \left.-A_{\alpha\beta}(x_0, \varphi(x_0), D' \varphi(x_0),D_n \underline u(x_0))\right ] - B(x_0, \varphi(x_0), D'\varphi(x_0),D_nu(x_0))\\
        \!\!&\!\!\displaystyle  + B(x_0, \varphi(x_0), D' \varphi(x_0),D_n \underline u(x_0)) \\
    \ge \!\!&\!\!\displaystyle  -D_n(u-\underline u)(x_0)\left \{G_{x_0}^{\alpha\beta}D_\alpha\gamma_\beta(x_0)+[G_{x_0}^{\alpha\beta}D_{p_n}A_{\alpha\beta}+ D_{p_n}B](x_0, \varphi(x_0), D'\varphi(x_0),D_nu(x_0))\right \},
    \end{array}
    \end{equation}
where the regularity of $A$ and the convexity of $B$ with respect to $p$ are used in the last inequality. Assume that $\sigma: = h(y)= \min\limits_{x\in \partial\Omega}h(x)$, then $\sigma$ is a positive constant. Since $g(x_0)< h(x_0)/2$, we then have
    \begin{equation}\label{g-h 2}
        g(x_0) - h(x_0) < - \frac{1}{2}\sigma.
    \end{equation}
Since $\underline u$ can be regarded as a strict subsolution near the boundary, we have
    \begin{equation}\label{D n u - underline u}
        0< D_{n}(u-\underline u)(x_0) \le \kappa,
    \end{equation}
for a positive constant $\kappa$ depending on $\sup |D\underline u|$ and $\sup |Du|$. From \eqref{g-h 1}, \eqref{g-h 2} and \eqref{D n u - underline u}, we derive
    \begin{equation}\label{vartheta x0 > 0}
        G_{x_0}^{\alpha\beta}D_\alpha\gamma_\beta(x_0)+[G_{x_0}^{\alpha\beta}D_{p_n}A_{\alpha\beta}+ D_{p_n}B](x_0, \varphi(x_0), D'\varphi(x_0),D_nu(x_0)) \ge \frac{\sigma}{2\kappa}>0.
    \end{equation}
Let
    \begin{equation}\label{vartheta}
        \vartheta(x):=G_{x_0}^{\alpha\beta}\mathfrak{C}_{\alpha\beta}(x)+[G_{x_0}^{\alpha\beta}D_{p_n}\mathcal{A}_{\alpha\beta}+ D_{p_n}B](x,\varphi(x),\nabla \varphi(x),
        -D_\gamma u(x_0)),
    \end{equation}
then from \eqref{vartheta x0 > 0}, we have $\vartheta(x_0)\ge \frac{\sigma}{2\kappa}>0$.
Since $\vartheta(x)$ is smooth near $\partial \Omega$, we can have
    \begin{equation}\label{coefficient>0}
        \vartheta(x)\geq c>0,\ \ {\rm on}  \  \mathcal{N}\cap \partial\Omega,
    \end{equation}
for some small positive constant $c$. From the regularity condition of $A$, we observe that $\mathcal{A}_{\alpha\beta}$ is convex with respect to
$p_n$, for $1\le \alpha, \beta \le n-1$. Therefore, we have
    \begin{equation}\label{Aconvex}
        \begin{array}{ll}
            \!\!&\!\!\mathcal{A}_{\alpha\beta}(x,\varphi(x),\nabla \varphi(x), -D_\gamma
            u(x_0))-\mathcal{A}_{\alpha\beta}(x,\varphi(x),\nabla \varphi(x), -D_\gamma
            u(x))\\
            \leq \!\!&\!\! D_{p_n}\mathcal{A}_{\alpha\beta}(x,\varphi(x),\nabla \varphi(x), -D_\gamma
            u(x_0))(D_\gamma u(x)-D_\gamma
            u(x_0)),
        \end{array}
    \end{equation}
on $\partial\Omega$, for $1\le \alpha, \beta \le n-1$.
From \eqref{nabla2u-varphi}, \eqref{starting point 1}, \eqref{Aconvex} and the convexity of $B$ in $p$, we have, on $\partial\Omega$,
    \begin{equation}\label{D gamma u}
        \begin{array}{ll}
            \!\!&\!\!\displaystyle \vartheta(x) \left[D_\gamma (u-\varphi)(x) - D_{\gamma}(u-\varphi)(x_0)\right]\\
            \ge \!\!&\!\!\displaystyle
            G_{x_0}^{\alpha\beta}\left[D_\gamma(u-\varphi)(x_0)(\mathfrak{C}_{\alpha\beta}(x_0)-\mathfrak{C}_{\alpha\beta}(x))+\nabla_{\alpha\beta}\varphi(x_0)
            -\nabla_{\alpha\beta}\varphi(x)\right.\\
            \!\!&\!\!\displaystyle +\mathcal{A}_{\alpha\beta}(x,\varphi(x),\nabla \varphi(x),-D_\gamma
            u(x_0))-\mathcal{A}_{\alpha\beta}(x_0,\varphi(x_0),\nabla \varphi(x_0),-D_\gamma u(x_0))\\
            \!\!&\!\!\displaystyle \left.+D_{p_n}\mathcal{A}_{\alpha\beta}(x,\varphi(x),\nabla \varphi(x),
            -D_\gamma u(x_0))(D_\gamma \varphi(x_0)-D_\gamma \varphi(x)\right]\\
            \!\!&\!\!\displaystyle+B(x,\varphi(x),\nabla \varphi(x),-D_\gamma u(x_0))-B(x_0,\varphi(x_0),\nabla \varphi(x_0),-D_\gamma u(x_0))\\
            \!\!&\!\!\displaystyle+D_{p_n}B(x,\varphi(x),\nabla \varphi(x),-D_\gamma u(x_0))(D_\gamma \varphi(x_0)-D_\gamma \varphi(x))\\
            := \!\!&\!\! \displaystyle \Theta(x).
        \end{array}
    \end{equation}
By \eqref{coefficient>0} and \eqref{D gamma u}, we have
    \begin{equation}\label{Dnu<Theta}
        D_\gamma (u-\varphi)(x) - D_{\gamma}(u-\varphi)(x_0) \ge \vartheta^{-1}(x) \Theta(x),\ \ {\rm on} \ \mathcal{N}\cap \partial\Omega,
    \end{equation}
namely,
    \begin{equation}
    \begin{array}{ll}
         \!\!&\!\!\displaystyle \gamma_n(x)D_n (u-\varphi)(x) + D_n(u-\varphi)(x_0) \\
    \ge  \!\!&\!\!\displaystyle \vartheta^{-1}(x) \Theta(x) - \sum_{\alpha=1}^{n-1}\gamma_\alpha (x)D_{\alpha}(u-\varphi)(x),\ \ {\rm on} \ \mathcal{N}\cap \partial\Omega.
    \end{array}
    \end{equation}
Since $\gamma_n\in [-1, -1/2)$ in $\mathcal{N}\cap \partial\Omega$, we have
    \begin{equation}\label{D n 1}
    \begin{array}{ll}
         \!\!&\!\!\displaystyle D_n (u-\varphi)(x) + \gamma^{-1}_n(x) D_n(u-\varphi)(x_0) \\
    \le  \!\!&\!\!\displaystyle \gamma^{-1}_n(x)\vartheta^{-1}(x) \Theta(x) - \gamma^{-1}_n(x)\sum_{\alpha=1}^{n-1}\gamma_\alpha (x)D_{\alpha}(u-\varphi)(x),\ \ {\rm on} \ \mathcal{N}\cap \partial\Omega.
    \end{array}
    \end{equation}
From the form of the function $\Theta(x)$ in \eqref{D gamma u}, since $\Theta(x_0)=0$, we have
    \begin{equation}\label{l + cx2}
    \gamma^{-1}_n(x)\vartheta^{-1}(x) \Theta(x) \le \ell(x-x_0) + C|x-x_0|^2,\quad  {\rm on} \ \mathcal{N}\cap \partial\Omega,
    \end{equation}
where $\ell$ is a linear function of $x-x_0$ with $\ell(0)=0$, and the constant $C$ depends on $\Omega, A, B, \varphi$ and $|u|_{1;\Omega}$. Since $-\gamma^{-1}_n\in [1,2)$, $\gamma_\alpha(x_0)=0$ and $D_\alpha(u-\varphi)(x_0)=0$ for $\alpha=1, \cdots, n-1$, we have
    \begin{equation}\label{cx2}
    \begin{array}{ll}
        \!\!&\!\!\displaystyle - \gamma^{-1}_n(x)\sum_{\alpha=1}^{n-1}\gamma_\alpha (x)D_{\alpha}(u-\varphi)(x) \\
    =   \!\!&\!\!\displaystyle - \gamma^{-1}_n(x)\sum_{\alpha=1}^{n-1}[\gamma_\alpha (x)-\gamma_\alpha (x_0)][D_{\alpha}(u-\varphi)(x)-D_{\alpha}(u-\varphi)(x_0)] \\
    \le \!\!&\!\!\displaystyle C|x-x_0|^2, \quad\quad  {\rm on} \ \mathcal{N}\cap \partial\Omega,
    \end{array}
    \end{equation}
where the constant $C$ depends on $\Omega$, $\varphi$ and $M^\prime_2$. From \eqref{D n 1}, \eqref{l + cx2} and \eqref{cx2}, we have
    \begin{equation}\label{2.59}
    v(x):=D_n (u-\varphi)(x) + \gamma^{-1}_n(x) D_n(u-\varphi)(x_0) - \ell(x-x_0) \le C|x-x_0|^2, \quad  {\rm on} \ \mathcal{N}\cap \partial\Omega,
    \end{equation}
where the constant $C$ depends on $\Omega, A, B, \varphi, |u|_{1;\Omega}$ and $M^\prime_2$.
By extending $\varphi$ and $\gamma$ smoothly to the interior near the boundary to be constant in the normal direction, the function $v$ in \eqref{2.59} is extended to $\Omega\cap B_\delta(x_0)$ for some small $\delta$ such that
    \begin{equation}
    B_\delta(x_0) \cap \partial \Omega \subset \mathcal{N}\cap \partial\Omega.
    \end{equation}
By calculations, we have
    \begin{equation}
        |\mathcal{L}v| \leq C(1+\mathscr{T}),\ \ {\rm in} \ \Omega\cap B_\delta(x_0),
    \end{equation}
where the differentiated equation \eqref{LDku} for $k=n$ is used.
Recalling the barrier function $\tilde \eta$ in \eqref{final barrier} with $a \gg b \gg 1$ and $|x|^2$ replaced by $|x-x_0|^2$, we have
    \begin{equation}
    \begin{array}{cl}
    \displaystyle \mathcal{L}\tilde \eta \le -\frac{a\epsilon_1}{2} (1+\mathscr{T}), & \quad {\rm in} \ \Omega\cap B_{\delta}(x_0),\\
    \displaystyle        \tilde \eta  = b|x-x_0|^2, & \quad {\rm on} \ \partial\Omega\cap B_\delta(x_0), \\
    \displaystyle        \tilde \eta  > b\delta^2,  & \quad {\rm on} \ \Omega \cap \partial B_\delta(x_0).
    \end{array}
    \end{equation}
Therefore, for $a\gg b \gg 1$, we have
    \begin{equation}
    \begin{array}{cl}
    \displaystyle \mathcal{L}v \ge \mathcal{L}\tilde \eta, & \quad {\rm in} \ \Omega\cap B_{\delta}(x_0),\\
    \displaystyle            v \le \tilde \eta,            & \quad {\rm on} \ \partial(\Omega\cap B_\delta(x_0)), \\
    \displaystyle            v=\tilde \eta =0,             & \quad {\rm at} \ x_0.
    \end{array}
    \end{equation}
Then the maximum principle leads to
    \begin{equation}
        D_n v\le D_n \tilde \eta,\ \ {\rm at}\ x_0,
    \end{equation}
and hence
    \begin{equation}\label{Dnnu up}
        D_{nn}u(x_0)\le C,
    \end{equation}
where the constant $C$ depends on $\Omega, A, B, \varphi, \underline u, |u|_{1;\Omega}$ and $M^\prime_2$.
Therefore, we have obtained the upper bound for all eigenvalues of $M[u](x_0)$. Then by F3, $\lambda(M[u](x_0))$ is contained in a compact subset of $\tilde \Gamma$.
Hence for sufficiently large $R$, we have
	\begin{equation}
		g(x_0)=\min_{x\in \partial\Omega}g(x) >0.
	\end{equation}

Overall, from cases 1 and 2, we have obtained a positive lower bound $c_0>0$ for the function $g(x)$ defined in \eqref{g function} on $\partial\Omega$.
By Lemma 1.2 in \cite{CNS-Hessian} and using \eqref{M2 prime estimate}, there exists a constant $R_0\ge R$ such that if $w_{\gamma\gamma}(x_0)>R_0$, we have
\begin{equation}\label{f lambda Mu inq}
f(\lambda (M[u](x_0))) \ge f_{w_{\gamma\gamma}(x_0)}(\lambda^\prime (M^\prime [u])(x_0) ) - \frac{c_0}{2},
\end{equation}
where 
\begin{equation}\label{f lambda Mu def}
f(\lambda (M[u])): = F(M[u]). 
\end{equation}
Note that since $F$ is orthogonally invariant, $f_R (\lambda^\prime(M^\prime[u]))$ and $f(\lambda (M[u](x_0)))$ in \eqref{f infty} and \eqref{f lambda Mu def} are well defined.
Then if $w_{\gamma\gamma}(x_0)>R_0\ge R$, from \eqref{f infty}, \eqref{g function}, \eqref{f lambda Mu inq}, \eqref{f lambda Mu def} and $g(x_0)\ge c_0>0$, we have
\begin{equation}
F(M[u](x_0)) - B[u](x_0) \ge \frac{c_0}{2} >0,
\end{equation}
which leads to a contradiction with \eqref{1.1}. Consequently, we have $w_{\gamma\gamma}(x_0) \le R_0$, which leads to
\begin{equation}\label{u gamma gamma x0}
D_{\gamma\gamma}u(x_0)\le C,
\end{equation}
for some constant $C$.

Since $x_0\in \partial\Omega$ is a point where the function $g$ in \eqref{g function} is minimized over $\partial\Omega$, we can repeat the argument from \eqref{f lambda Mu inq} to \eqref{u gamma gamma x0} at any boundary point $x\in \partial\Omega$ to get $D_{\gamma\gamma}u \le C$ on $\partial \Omega$ for some constant $C$.
Then together with the lower bound (from the ellipticity), we finally get the pure normal second derivative estimate on the boundary,
    \begin{equation}\label{pure normal}
        |D_{\gamma\gamma}u|\le C, \quad {\rm on} \ \partial\Omega,
    \end{equation}
where the constant $C$ depends on $\Omega, A, B, \varphi, \underline u, |u|_{1;\Omega}$ and $M^\prime_2$. 

\begin{remark}
	The {\it a priori} pure normal second derivative estimate \eqref{pure normal} on $\partial\Omega$ is treated using the idea in \cite{Tru1995}. The proof in \cite{Tru1995} is divided into the bounded case and the unbounded case, which include concrete examples of the Hessian quotient operator and $k$-Hessian operator respectively. In \cite{Tru1995}, the bounded case is proved by using a limit function, namely replacing $g(x)$ in \eqref{g function} by $g(x):=\lim\limits_{R\rightarrow +\infty} f_R(\lambda^\prime(M^\prime [u])) - B[u]$. In this paper, the estimate \eqref{pure normal} is proved in a uniform package, which is different from \cite{Tru1995}.
\end{remark}

Combining the estimates \eqref{M2 prime estimate} and \eqref{pure normal}, we now have obtained the second derivative bound on the boundary
    \begin{equation}\label{2nd bd on boundary}
        \sup_{\partial\Omega} |D^2 u| \le C,
    \end{equation}
where the constant $C$ depends on $\Omega, A, B, \varphi, \underline u$ and $|u|_{1;\Omega}$.

With the boundary estimate \eqref{2nd bd on boundary}, we can now give the proof of Theorem \ref{Th1.1}.
\begin{proof}[Proof of Theorem \ref{Th1.1}]
Since $\mathcal{F}$ satisfies F1-F3, $A$ is regular, $B$ is convex in $p$, and $\underline u$ is an admissible subsolution, from case (ii) of Theorem 3.1 in \cite{JT-oblique-II}, we have the global second derivative estimate
\begin{equation}\label{global 2nd bound regular A}
  \sup_{\Omega}|D^2u|\le C(1+ \sup_{\partial\Omega}|D^2u|),
\end{equation}
where the constant $C$ depends on $A, B, F, \Omega, \underline u$ and $|u|_{1;\Omega}$. The full second derivative estimate \eqref{full 2nd estimate} then follows from the estimates \eqref{2nd bd on boundary} and \eqref{global 2nd bound regular A}. We now complete the proof of Theorem \ref{Th1.1}.
\end{proof}

\subsection*{A strengthened barrier and its applications}\label{Section 2.1}

To end this section, we shall strengthen the key barrier construction in \eqref{barrier in Lemma 2.1} in Lemma \ref{Lemma 2.1}, which makes the proof of the pure normal derivative estimate on the boundary a bit simpler. This barrier also provides an alternative proof of the mixed tangential-normal derivative estimate on the boundary.
Such a barrier is achieved by using $|\delta u|^2$, where $\delta u$ is the tangential gradient. The idea has already been used in the uniformly elliptic case in \cite{Tru1983}, in the case of curvature equations in \cite{IV1990, LT1994}, and in the case of general fully nonlinear equations on Riemannian manifolds in \cite{Guan2014}.

In the proof of the pure normal second derivative estimate on $\partial\Omega$, immediately after \eqref{D gamma u}, we define
    \begin{equation}\label{aux alt}
        v(x):=\vartheta(x)[D_\gamma (u-\varphi)(x) - D_{\gamma}(u-\varphi)(x_0)]-\Theta(x),
    \end{equation}
where the functions $\vartheta(x)$ and $\Theta(x)$ are defined in \eqref{vartheta} and \eqref{D gamma u}, respectively.
Then by \eqref{D gamma u}, we have
    \begin{equation}\label{v ge 0 boundary}
        v(x) \ge 0, \quad {\rm on}\ \partial\Omega.
    \end{equation}
By extending $\varphi$ and $\gamma$ smoothly to $\bar\Omega$ such that $|\gamma|=1$, using \eqref{LDku} and the orthogonal invariance of $\mathcal{F}$, we have
    \begin{equation}\label{inequality with additional term}
        |\mathcal{L} v|\le C\left (1 + \mathscr{T}^*\right ), \quad {\rm in} \ \Omega,
    \end{equation}
where 
	\begin{equation}\label{T*}
	\mathscr{T}^*=\sum_{i=1}^n F^{ii}(1+|w_{ii}|).
	\end{equation}

Comparing \eqref{inequality with additional term} with the standard inequality of the form \eqref{standard inequality}, there is an additional term $\sum_{i=1}^n F^{ii}|w_{ii}|$. We need to modify the barrier function in Lemma \ref{Lemma 2.1} to derive a barrier inequality which can control the additional term. For this purpose, we assume in addition that $\mathcal{F}$ is orthogonally invariant, and satisfies
\begin{equation}\label{3.053}
	|r\cdot F_r| \le O(1) (\mathscr{T}(r) +  |F(r)|),
\end{equation}
as $|r| \rightarrow \infty$, uniformly for $F(r) >a$, for any $a >a_0$.
The condition \eqref{3.053} is a combination of the conditions (3.24) and (3.54) in \cite{JT-oblique-I}. Note that \eqref{3.053} is satisfied if F2 holds and either $a_0$ is finite or F4 in \cite{JT-oblique-I} holds, (or trivially if $\mathcal{F}$ is homogeneous).
We now formulate the following lemma.

\begin{Lemma}\label{Lemma 2.2}
Under the assumptions of Lemma \ref{Lemma 2.1}, assume also that $\mathcal{F}$ is orthogonally invariant and satisfies \eqref{3.053}. Then there exist positive constants $K$ and $\epsilon_2$, depending on $\Omega, A, B, \underline u$ and $|u|_{1;\Omega}$, such that
	\begin{equation}\label{new barrier}
	\mathcal{L} \left [ e^{K(\underline u-u)} + \frac{\epsilon_2}{2} |\delta u|^2 \right ] \ge \epsilon_2 \left ( 1+ \mathscr{T}^*\right ), \quad {\rm in} \ \Omega,
	\end{equation}
where $\mathcal{L}$ is the linearized operator defined in \eqref{mathcal L def}, $\delta u = Du -(D_\gamma u)\gamma$ denotes the tangential gradient of $u$, and $\mathscr{T}^*$ is defined in \eqref{T*}.
\end{Lemma}

Here the unit vector field $\gamma$ in $\Omega$ in Lemma \ref{Lemma 2.2} is extended smoothly from the unit normal vector $\gamma$ on $\partial\Omega$.

\begin{proof}[Proof of Lemma \ref{Lemma 2.2}.]
In view of the estimate \eqref{barrier in Lemma 2.1} in Lemma \ref{Lemma 2.1}, we only need to estimate $\frac{1}{2}\mathcal{L}|\delta u|^2$. By calculations, we have
    \begin{equation}\label{L delta u}
    \begin{array}{rl}
          \displaystyle\frac{1}{2}\mathcal{L}|\delta u|^2 =  \!\!&\!\!\displaystyle F^{ij} \left ( a_{kl}u_{ik}u_{jl} + \tilde \beta_{ik}u_{jk} \right ) + \delta_ku \mathcal{L}u_k - u_ku_l \left[ \gamma_l \mathcal{L}\gamma_k + F^{ij}(D_i\gamma_k) (D_j \gamma_l)\right ]\\
       \ge \!\!&\!\!\displaystyle F^{ij} \left ( a_{kl}u_{ik}u_{jl} + \tilde \beta_{ik}u_{jk} \right )  - C(1+\mathscr{T}),
    \end{array}
    \end{equation}
where $a_{kl}=\delta_{kl}-\gamma_k\gamma_l$, $\tilde \beta_{ik}=-2u_l(\gamma_kD_i\gamma_l + \gamma_l D_i\gamma_k)$, $C$ is a constant depending on $\Omega, A, B, |u|_{1;\Omega}$ and $|\gamma|_{1;\Omega}$, and \eqref{LDku} is used to obtain the inequality. Note that the estimate \eqref{L delta u} can also be obtained directly from (3.8) in \cite{JT-oblique-I}. At any fixed point $x\in \Omega$, by choosing coordinates so that $M[u]=\{w_{ij}\}$ is diagonal at the point $x$. From the orthogonal invariance of $\mathcal{F}$, we can estimate the first term on the right hand side of \eqref{L delta u},
    \begin{equation}\label{3.051}
        \begin{array}{rl}
                \!\!&\!\!\displaystyle F^{ij} \left ( a_{kl}u_{ik}u_{jl} + \tilde \beta_{ik}u_{jk} \right )  \\
            =   \!\!&\!\!\displaystyle F^{ij} \left [ a_{kl}(w_{ik}+A_{ik})(w_{jl}+A_{jl}) + \tilde \beta_{ik}(w_{jk}+A_{jk}) \right ] \\ 
            =   \!\!&\!\!\displaystyle F^{ij} \left [ a_{kl} w_{ik}w_{jl} + (\tilde\beta_{ik} + 2a_{kl}A_{il})w_{jk} + (a_{kl}A_{ik}A_{jl} + \tilde \beta_{ik}A_{jk})\right ] \\
                \ge \!\!&\!\! \displaystyle F^{ii}(1-\gamma_i^2) w_{ii}^2 - C\left (\mathscr{T} + F^{ii}|w_{ii}| \right ),
        \end{array}
    \end{equation}
where the  constant $C$ depends on $\Omega, A, |u|_{1;\Omega}$ and $|\gamma|_{1;\Omega}$.
Since $\gamma$ is a unit vector field, we can fix $k$ so that $\gamma_k^2 = \max\limits_i \gamma_i^2\ge \frac{1}{n}$. Then we have
	\begin{equation}\label{1-gammai2}
	1-\gamma_i^2\ge \frac{n-1}{n}, \quad {\rm for} \ i\neq k.
	\end{equation}
By successively using the reverse triangle inequality and the triangle inequality, we have 
	\begin{equation}\label{3.052}
	\begin{array}{rl}
	\displaystyle |\sum_{i=1}^n F^{ii}w_{ii}| = \!\!&\!\!\displaystyle |F^{kk}w_{kk}+\sum_{i\neq k} F^{ii}w_{ii}| \\
	                                                            \ge\!\!&\!\!\displaystyle |F^{kk}w_{kk}|-|\sum_{i\neq k} F^{ii}w_{ii}| \\
	                                                            \ge\!\!&\!\!\displaystyle F^{kk}|w_{kk}|-\sum_{i\neq k} F^{ii}|w_{ii}| \\
	                                                              = \!\!&\!\!\displaystyle \sum_{i=1}^nF^{ii}|w_{ii}| - 2 \sum_{i\neq k} F^{ii}|w_{ii}|.
	\end{array}
	\end{equation}
From \eqref{3.053}, \eqref{1-gammai2}, \eqref{3.052} and Cauchy's inequality, we have
    \begin{equation}\label{3.054}
        \begin{array}{rl}
           \displaystyle \sum_{i=1}^nF^{ii}|w_{ii}| \!\!&\!\!\displaystyle \le \frac{(n-1)\epsilon}{n}\sum_{i\neq k}F^{ii}w_{ii}^2 +\frac{n}{(n-1)\epsilon}\mathscr{T}+ \mu (1+\mathscr{T}+|F(M[u])|) \\
           \!\!&\!\!\displaystyle \le \epsilon \sum_{i\neq k}F^{ii}(1-\gamma_i^2) w_{ii}^2 + \frac{2}{\epsilon}\mathscr{T}+ \mu (1+\mathscr{T}+|B|) \\
           \!\!&\!\!\displaystyle \le \epsilon \sum_{i=1}^nF^{ii}(1-\gamma_i^2) w_{ii}^2 + \frac{2}{\epsilon}\mathscr{T}+ \mu (1+\mathscr{T}+|B|) ,
        \end{array}
    \end{equation}
for any constant $\epsilon>0$, and some positive constant $\mu$. Namely, 
	\begin{equation}\label{3.054'}
	\sum_{i=1}^nF^{ii}(1-\gamma_i^2) w_{ii}^2 \ge \frac{1}{\epsilon} \sum_{i=1}^nF^{ii}|w_{ii}| - \frac{1}{\epsilon} \left[\mu +(\mu+\frac{2}{\epsilon})\mathscr{T}+ \mu|B|\right ]
	\end{equation}
holds for any constant $\epsilon>0$.
Combining \eqref{L delta u}, \eqref{3.051} and \eqref{3.054'}, we have
    \begin{equation}\label{3.055}
        \frac{1}{2}\mathcal{L}|\delta u|^2 \ge  \left( \frac{1}{\epsilon} - C \right)\sum_{i=1}^nF^{ii}|w_{ii}| - C (\mathscr{T}+1) - \frac{1}{\epsilon} \left[\mu +(\mu+\frac{2}{\epsilon})\mathscr{T}+ \mu|B|\right ],
    \end{equation}
for any constant $\epsilon>0$, where $C$ is a further constant depending on $\Omega, A, B, |u|_{1;\Omega}$ and $|\gamma|_{1;\Omega}$. Using \eqref{barrier in Lemma 2.1} in Lemma \ref{Lemma 2.1} and \eqref{3.055}, we have
	\begin{equation}\label{3.056}
	\begin{array}{rl}
	     \!\!&\!\! \displaystyle \mathcal{L} \left [ e^{K(\underline u-u)} + \frac{\epsilon_2}{2} |\delta u|^2 \right ] \\
	\ge\!\!&\!\! \displaystyle \epsilon_1(1+\mathscr{T}) + \epsilon_2  \left( \frac{1}{\epsilon} - C\right)\sum_{i=1}^nF^{ii}|w_{ii}| - C \epsilon_2 (\mathscr{T}+1) - \frac{\epsilon_2}{\epsilon} \left[\mu +(\mu+\frac{2}{\epsilon})\mathscr{T}+ \mu|B|\right ] \\
	\ge\!\!&\!\! \displaystyle (\epsilon_1-\epsilon_2 C^\prime) (1+\mathscr{T}) + \epsilon_2 \sum_{i=1}^nF^{ii}|w_{ii}|, 
	\end{array}
	\end{equation}
by fixing $\epsilon=\frac{1}{1+C}$, where the constant $C^\prime$ depends on $\mu, \Omega, A, B, |u|_{1;\Omega}$ and $|\gamma|_{1;\Omega}$.
By choosing $\epsilon_2 =\frac{\epsilon_1}{2\max\{C^\prime, 1\}}$ in \eqref{3.056}, we get the desired estimate \eqref{new barrier} and complete the proof of Lemma \ref{Lemma 2.2}.
\end{proof}

To apply Lemma \ref{Lemma 2.2} for the pure normal second derivative estimate on $\partial\Omega$, we need to make a slight modification of the function in \eqref{new barrier}. Let
	\begin{equation}\label{Phi def}
	\Phi := e^{K(\underline u-u)} -1+ \frac{\epsilon_2}{2} |\delta (u-\underline u)|^2,
	\end{equation}
where the constants $K$ and $\epsilon_2$ are the same as in \eqref{new barrier}. By directly using (3.8) in \cite{JT-oblique-I}, we can also obtain an estiamte
	\begin{equation}
	\frac{1}{2}\mathcal{L}|\delta(u-\underline u)|^2 \ge F^{ij} \left ( a_{kl}u_{ik}u_{jl} + \tilde \beta_{ik}u_{jk} \right )  - C(1+\mathscr{T}),
	\end{equation}
where $a_{kl}$ and $\tilde \beta_{jk}$ are the same as in \eqref{L delta u}, $C$ is a further constant depending on $\Omega, A, B, |u|_{1;\Omega}$ and $|\gamma|_{1;\Omega}$.
Therefore, following the steps in the proof of \eqref{new barrier}, it is readily checked that
	\begin{equation}\label{Phi in}
	\mathcal{L}\Phi \ge \epsilon_2 \left ( 1+ \mathscr{T}^*\right ), \quad {\rm in} \ \Omega.
	\end{equation}
Moreover, it is obvious that
	\begin{equation}\label{Phi on}
	\Phi =0, \quad {\rm on} \ \partial\Omega.
	\end{equation}	
From \eqref{v ge 0 boundary}, \eqref{inequality with additional term}, \eqref{Phi in} and \eqref{Phi on}, we have 
	\begin{equation}
         \begin{array}{rl}
          \mathcal{L}(v-\tau \Phi) \le 0, & \quad {\rm in} \ \Omega,\\
          v-\tau \Phi \ge 0, & \quad {\rm on} \ \partial\Omega,
         \end{array}
	\end{equation}
for sufficiently large positive constant $\tau$, which leads to
    \begin{equation}
       v - \tau\Phi\ge 0, \quad {\rm in} \ \Omega.
    \end{equation}
Since $v - \tau\Phi= 0$ at $x_0\in \partial\Omega$, we have
   \begin{equation}\label{Dnv alt}
   D_n v(x_0)\ge \tau D_n \Phi(x_0)\ge - C, 
   \end{equation}
where the constant $C$ depends on $\Omega, A, B, \underline u, |u|_{1;\Omega}$ and $M_2^\prime$.
Using \eqref{aux alt} and $\vartheta(x_0)\ge \frac{\sigma}{2\kappa}>0$, we have from \eqref{Dnv alt} that
    \begin{equation}\label{Dnnu alt}
        D_{nn}u(x_0) \le C,
    \end{equation}
where the constant $C$ depends on $\Omega, A, B, \underline u, |u|_{1;\Omega}$ and $M_2^\prime$. Note that $\underline u$ in the last term of \eqref{Phi def} can be replaced by $\varphi$, in this case the constant $C$ in \eqref{Dnnu alt} depends also on $\varphi$. We are now in the same position as \eqref{Dnnu up}. We shall omit the rest of the proof for the pure normal derivative estimate on $\partial\Omega$, since it is the same as the previous argument.

We remark that once the barrier \eqref{new barrier} or \eqref{Phi def} is constructed, for the pure normal derivative bound on $\partial\Omega$, we do not need to make detailed local analysis from \eqref{Dnu<Theta} to \eqref{2.59}. In this sense, using such a barrier \eqref{new barrier} or \eqref{Phi def} is a bit simpler and more direct than using the previous barrier \eqref{final barrier} in the course of pure normal derivative estimate.

Next, we show an alternative proof of the mixed tangential-normal derivative estimate on $\partial\Omega$, which is immediate from the barrier \eqref{new barrier} or \eqref{Phi def}.
By F1, we know that $\mathscr{T}^*>0$. Then for the function $\Phi$ in \eqref{Phi def}, from \eqref{Phi in} we have $\mathcal{L}\Phi>0$ in $\Omega$. By the maximum principle, we have
	\begin{equation}
		\Phi\le 0, \ {\rm in} \ \Omega, \quad {\rm and} \ \  \Phi=0,\  {\rm on} \  \partial\Omega,
	\end{equation}
which leads to
	\begin{equation*}
	\begin{array}{cl}
		\pm \delta_i(u-\underline u) \le \sqrt{2 \left(1-e^{K(\underline u-u)}\right) / \epsilon_2}, & \quad {\rm in} \ \Omega, \\
		\pm \delta_i(u-\underline u) =    0, & \quad {\rm on} \  \partial\Omega,
	\end{array}
	\end{equation*}
for $i=1, \cdots, n$. Then we have
	\begin{equation}\label{mixed est using delta}
		\pm D_\gamma \delta_i (u-\underline u) \ge D_\gamma \sqrt{2 \left(1-e^{K(\underline u-u)}\right) / \epsilon_2},\quad {\rm on} \  \partial\Omega,
	\end{equation}
for $i=1, \cdots, n$, where $\gamma$ is the unit outer normal vector field on $\partial\Omega$. Hence, from \eqref{mixed est using delta} we have
	\begin{equation}\label{mixed deri est}
		|D_{\tau\gamma}u|\le C,\quad {\rm on} \  \partial\Omega,
	\end{equation}
for any unit tangential vector field $\tau$ on $\partial\Omega$.

\begin{remark}
Under the additional assumptions that $\mathcal{F}$ is orthogonally invariant and \eqref{3.053} holds, we derive the strengthened barrier inequality \eqref{new barrier}, and further provide alternative proofs of the mixed tangential-normal derivatives and the pure normal derivatives on $\partial\Omega$.  When F2 holds and $a_0$ is finite, condition \eqref{3.053} is automatically satisfied, (see (1.10) in \cite{JT-oblique-I}).
Note that in Theorem \ref{Th1.1}, we already assumed that $\mathcal{F}$ is orthogonally invariant and F2 holds. Therefore, when $a_0$ is finite, the pure normal derivative estimate \eqref{Dnnu alt} and the mixed tangential-normal derivative estimate \eqref{mixed deri est} can be used directly to obtain the full second order derivative estimate \eqref{full 2nd estimate} in Theorem \ref{Th1.1}.
\end{remark}

\begin{remark}
Note that in the Riemannian manifold case, we would encounter this type of estimate \eqref{inequality with additional term} in the course of estimating the mixed tangential-normal derivatives and the pure normal derivatives on the boundary, where the additional term $\sum_{i=1}^n F_{ii}|w_{ii}|$ can not be avoided. Therefore, such kind of barrier in \eqref{new barrier} is useful in the mixed tangential-normal derivative estimate and the pure normal derivative estimate on the boundary for the Riemannian manifold case. For the Dirichlet problem \eqref{1.1}-\eqref{1.2} on Riemannian manifold, we refer the reader to \cite{Guan2014, GJ2015, GJ2016} for more detailed discussions.
\end{remark}

\section{Gradient estimates and existence theorem}\label{Section 3}
\vskip10pt

In this section, we discuss the gradient estimates for admissible solutions under appropriate growth conditions of $A$ and $B$ with respect to $p$, and then combine all the derivative estimates to prove the existence result, Theorem \ref{Th1.2}. The alternative existence results in Corollary \ref{Cor 3.1} is properly explained.

When $\mathcal{F}$ satisfies F7, we have the global gradient estimate under the growth conditions \eqref{structure condition 1} and \eqref{structure condition 2} for $A$ and $B$.

\begin{Theorem}\label{Th3.1}
Assume that $\mathcal{F}$ is orthogonally invariant satisfying F1, F3 and F7, $A, B\in C^1(\bar \Omega\times \mathbb{R}\times \mathbb{R}^n)$ satisfying \eqref{structure condition 1} and \eqref{structure condition 2}, $B>a_0$, $u\in C^3(\Omega)\cap C^2(\bar \Omega)$ is an admissible solution of equation \eqref{1.1}.
Then we have the gradient estimate 
\begin{equation}\label{global 1nd bound regular A}
\sup_{\Omega}|Du| \le C(1+\sup_{\partial\Omega}|Du|),
\end{equation}
where the constant $C$ depends on $F, A, B, \Omega$ and $|u|_{0;\Omega}$.
\end{Theorem}
The proof of Theorem \ref{Th3.1} here follows directly from the proof of Theorem 1.3(ii) in Section 3 of \cite{JT-oblique-I} and the remark of the case when discarding the boundary condition in Remark 3.1 in \cite{JT-oblique-I}. However it should be noted that our condition \eqref{structure condition 1} is written more generally than the corresponding conditions (3.31) and (3.33) in Remark 3.1(ii') in \cite{JT-oblique-I}. The replacement of (3.31) by the last two inequalities in  \eqref{structure condition 1} is immediate from (3.32) in \cite{JT-oblique-I} while the replacement of (3.33) by the corresponding inequality in
\eqref{structure condition 1} follows by examination the derivation of (3.42), in the case $g=|Du|^2$, and is readily seen by multiplying through inequality (3.38), (in the general case), by $u_i-\varphi\nu_i$.

 If we replace F7 by F2 in Theorem \ref{Th3.1} , we still need to assume, when $B$ is unbounded,
condition F5 in \cite{JT-oblique-I} with $b=\infty$, that is
\begin{itemize}
 \item[{\bf F5($\infty$)}:] 
 For a given constant $a>a_0$, there exists a constant $\delta_0>0$ such that $\mathscr{T}(r) \ge \delta_0$ if $a< F(r)$,
\end{itemize}
which is implied by F7. 

We also need some control from below on $r\cdot F_r$, as in condition (3.54) in \cite{JT-oblique-I}, namely
\begin{equation}\label{3.1}
r\cdot F_r \ge o(|\lambda_0(r)|) \mathscr{T}(r),
\end{equation}
as $\lambda_0(r) \rightarrow -\infty$, uniformly for $F(r) >a$, for any $a >a_0$, where $\lambda_0(r)$ denotes the minimum eigenvalue of $r$. Note that if F1-F3 hold with $a_0$ finite,  we have $r\cdot F_r\ge 0$, so \eqref{3.1} is trivially satisfied. 

We then have as an alternative to Theorem \ref{Th3.1},  
\begin{Theorem}\label{Th3.2}
Assume that $\mathcal{F}$ is orthogonally invariant satisfying F1-F3, F5($\infty$) and \eqref{3.1}, $A, B\in C^1(\bar \Omega\times \mathbb{R}\times \mathbb{R}^n)$ satisfying \eqref{structure condition 1}  and \eqref{structure condition 2}, with ``O'' replaced by ``o'' in \eqref{structure condition 1}, $B>a_0$, $u\in C^3(\Omega)\cap C^2(\bar \Omega)$ is an admissible solution of equation \eqref{1.1}.
Then we have the gradient estimate \eqref {global 1nd bound regular A}, where the constant $C$ depends on $F, A, B, \Omega$ and $|u|_{0;\Omega}$.
\end{Theorem}


\begin{proof}[Proof of Theorem \ref{Th3.2}.]

The proof of Theorem \ref{Th3.2} here follows from a slight modification of the proof of Theorem \ref{Th3.1}. The technical details are somewhat simpler as we can employ an auxiliary function of the form,
\begin{equation}\label{aux gb}
v: = |Du|^2 + \alpha M_1^2 \eta,
\end{equation}
in $\Omega$, where $\eta= u-u_0$,  $M_1=\sup\limits_{\Omega}|Du|$, $u_0 =  \inf \limits_\Omega u$ and $\alpha$ is a positive constant satisfying
\begin{equation}\label{restriction}
\alpha \le \frac{1}{2 \mathop{\rm osc}_{\Omega}\eta}.
\end{equation}
In place of (3.32) in \cite{JT-oblique-I}, we now obtain from our strengthening of \eqref{structure condition 1},
\begin{equation}
\mathcal {L} \eta \ge F^{ij}w_{ij} - C(1+\mathscr{T})(\omega|Du|^2+1),
\end{equation}
where $C$ is a positive constant and $\omega = \omega(|Du|)$ a positive decreasing function on $[0,\infty)$ tending to $0$ at infinity, depending on $A,B$ and $\Omega$. We consider the case that the maximum of $v$ occurs at a point $x_0\in \Omega$.
Following the proof of case (ii) of Theorem 1.3 in \cite{JT-oblique-I} with $g=|Du|^2$ and our simpler $\eta$, we obtain, in place of inequality (3.42) in \cite{JT-oblique-I},
\begin{equation}\label{push a negative w11}
w_{11} \le   -\frac{1}{2}\alpha M^2_1+C(\omega|Du|^2+1).
\end{equation}
Now we observe that the estimate \eqref{push a negative w11} is clearly applicable to the minimum eigenvalue $w_{kk}$ of $M[u]$ and moreover by F2 we must have $F^{kk} \ge \mathscr {T}/n$; (see \cite{Urbas1995} and Remark \ref{Rem 3.1} below). 
Retaining the term $\mathcal{E}_2^\prime = F^{ij}u_{ik}u_{jk}$ in  (3.9) in \cite{JT-oblique-I}, instead of using (3.43) in \cite{JT-oblique-I}, we now obtain at $x_0$, in place of inequality (3.45) in  \cite{JT-oblique-I}, using F5($\infty$), \eqref{3.1} and \eqref{restriction},
\begin{equation}\label{combining together}
\begin{array}{rl}
0 \ge \mathcal{L}v \ge \!\!&\!\!\displaystyle   \mathcal{E}_2^\prime - C(1+\mathscr{T}) (\omega|Du|^4 +1) \\
                                   \!\!&\!\!\displaystyle +  \alpha  M^2_1[ F^{ij}w_{ij} - C(1+\mathscr{T})(\omega|Du|^2 +1)] \\
                              \ge \!\!&\!\!\displaystyle [\frac{1}{n}\alpha^2  M^4_1 - C (\omega|Du|^4 +1)]\mathscr{T},
\end{array}
\end{equation}
and we conclude $M_1\le C$ as desired. 
\end{proof}

\begin{remark}\label{Rem 3.1}
The concavity F2 and orthogonal invariance of $\mathcal F$ to imply that if $F(r) = f(\lambda)$, where $\lambda = (\lambda_1, \cdots, \lambda_n)$ denote the eigenvalues of $r \in\Gamma$, then $D_if \le D_jf$ at any fixed point $\lambda$, where $\lambda_i\ge \lambda_j$. 
Indeed, by applying the mean value theorem to the function $g = D_if - D_jf$ at the points $\lambda$ and $\lambda^*$ where $\lambda^*$ is given by exchanging $\lambda_i$ and $\lambda_j$ in $\lambda$, we have
\begin{equation}\label{g lambda lambda*}
g(\lambda) - g(\lambda^*) = Dg(\hat \lambda) (\lambda-\lambda^*) = [D_{ii}f(\hat \lambda) + D_{jj}f(\hat \lambda) - 2D_{ij}f(\hat \lambda)](\lambda_i-\lambda_j) \le 0,
\end{equation}
where $\hat \lambda = \theta\lambda + (1-\theta)\lambda^*$ for some constant $\theta \in (0,1)$, F2 and $\lambda_i\ge \lambda_j$ are used to obtain the inequality. Since $g(\lambda^*) = - g(\lambda)$ holds by symmetry of $f$,  \eqref{g lambda lambda*} implies $g(\lambda)\le 0$, and hence $D_if(\lambda) \le D_jf(\lambda)$.
\end{remark}

With these {\it a priori} derivative estimates in Theorems \ref{Th1.1} and \ref{Th3.1}, we can now give the proof of the existence result, Theorem \ref{Th1.2}, using method of continuity.
\begin{proof}[Proof of Theorem \ref{Th1.2}]
First, we need to establish the solution bound and the full gradient bound.
The bounded viscosity solution $\bar u$ and the subsolution $\underline u$ can provide the solution bound, namly 
\begin{equation}\label{SB}
\underline u \le u\le \bar u, \ {\rm in} \ \bar\Omega.
\end{equation}
For the gradient estimate on $\partial\Omega$, the tangential derivatives of $u$ are given by the Dirichlet boundary condition and the inner normal derivative bound from below is controlled by using the subsolution $\underline u$. From the admissibility of $u$ and $\Gamma \subset \Gamma_1$, we have $\Delta u \ge {\rm trace}(A)$, which leads to an inner normal derivative estimate of $u$ from above on $\partial\Omega$, under quadratic structure conditions of ${\rm trace}(A)$ with respect to $p$, see proof of Theorem 14.1 in \cite{GTbook}. We then obtain the gradient estimate of $u$ on $\partial\Omega$,
\begin{equation}\label{BGB}
\sup_{\partial\Omega} |Du| \le C,
\end{equation}
where the constant $C$ depends on $\Omega, \varphi$ and $|\underline u|_{1;\Omega}$. If \eqref{structure condition 1} and \eqref{structure condition 2} hold, the global gradient estimate \eqref{global 1nd bound regular A} holds in Theorem \ref{Th3.1}. If $\Gamma=K^+$ and $A(x,p)\ge O(|p|^2)I$ as $|p|\rightarrow \infty$, uniformly for $x\in \Omega$, the global gradient estimate \eqref{global 1nd bound regular A} holds in Section 4 in \cite{JTY2013}. Combining the solution estimates \eqref{SB}, global gradient estimate \eqref{global 1nd bound regular A}, and the boundary gradient estimate \eqref{BGB}, we obtain
\begin{equation}\label{lower order estimate}
  \sup_{\Omega} |u| + \sup_{\Omega} |Du| \le C,
\end{equation}
where the constant $C$ depends on $F, A, B, \Omega, \varphi, \bar u$ and $|\underline u|_{1;\Omega}$.

Then from the lower order estimate \eqref{lower order estimate} and the second derivative estimate \eqref{full 2nd estimate}, we have uniform estimates in $C^2(\bar \Omega)$ for classical admissible solutions of the Dirichlet problems
    \begin{equation}\label{ut equation}
    \mathcal{F}[u] = tB(\cdot, Du)+ (1-t) \mathcal{F}[\underline u], \quad {\rm in} \ \Omega,
    \end{equation}
    \begin{equation}\label{ut boundary}
    u=\varphi, \quad {\rm on} \ \partial\Omega,
    \end{equation}
for $0\le t\le 1$, where $\underline u$ is a subsolution. From the Evans-Krylov estimates, (Theorem 17.26' in \cite{GTbook}), we have the H\"older estimate for second derivatives of the admissible solution to the Dirichlet problem \eqref{ut equation}-\eqref{ut boundary}. Then the existence follows from the method of continuity, (Theorem 17.8 in \cite{GTbook}), and the uniqueness  from the maximum principle.

Moreover, if $\Gamma=\Gamma_k$ for $k>n/2$, we have the continuity estimate $|u(x)-u(y)|\le C|x-y|^\alpha (R^{-\alpha} \mathop{\rm osc}\limits_{\Omega\cap B_R}u+1)$ in (i), (iii) of Lemma 3.1 in \cite{JT-oblique-I}. By combining this continuity estimate and the local gradient estimates, we can still obtain the gradient estimate by replacing  \eqref{structure condition 1} by \eqref{structure condition 3} and extending ``$o$'' to ``$O$'' in \eqref{structure condition 2}, (see the last part of Theorem 3.1 in \cite{JT-oblique-I}). We then obtain the existence and uniqueness of a classical admissible solution and complete the proof.
\end{proof}

With the alternative gradient estimate in Theorem \ref{Th3.2}, we state the following existence result as a corollary of Theorem \ref{Th1.2}.

\begin{Corollary}\label{Cor 3.1}
Assume that $\mathcal{F}$ is orthogonally invariant and satisfies F1-F3, F5($\infty$) and \eqref{3.1} in $\Gamma\subset \Gamma_1$, $\Omega$ is a bounded domain in $\mathbb{R}^n$ with $\partial\Omega\in C^4$, $A\in C^2(\bar \Omega\times \mathbb{R}^n)$ is regular in $\bar \Omega$, $B>a_0, \in C^2(\bar \Omega\times \mathbb{R}^n)$ is convex with respect to $p$. Assume there exist a bounded viscosity supersolution $\bar u$ and a subsolution $\underline u\in C^2(\bar \Omega)$ satisfying $\underline u = \varphi$ on $\partial\Omega$ with $\varphi\in C^4(\partial\Omega)$. Assume also \eqref{structure condition 1}  and \eqref{structure condition 2} hold, with ``O'' replaced by ``o'' in \eqref{structure condition 1}.
Then there exists a unique admissible solution $u\in C^3(\bar \Omega)$ of the Dirichlet problem \eqref{1.1}-\eqref{1.2}.
\end{Corollary}

If $a_0$ is finite, condition \eqref{3.1} in Corollary \ref{Cor 3.1} can be dispensed with as it is automatically satisfied. If $B$ is bounded, F5($\infty$) in Corollary \ref{Cor 3.1} can be replaced by F5. Recalling that F5 is implied by F1, F2 and F3 when $a_0$ is finite (see Section 4.2 in \cite{JT-oblique-I}), hence we can replace ``F1-F3, F5($\infty$) and \eqref{3.1}'' by ``F1-F3'' in Corollary \ref{Cor 3.1} in the case when $a_0$ is finite and $B$ is bounded. When $B=B(x)>a_0, \in C^2(\bar \Omega)$ for finite $a_0$, Corollary \ref{Cor 3.1} holds automatically with ``F1-F3, F5($\infty$) and \eqref{3.1}'' replaced by ``F1-F3'', (since $B$ is bounded and satisfies the convexity condition with respect to $p$).

\begin{remark}
In Theorems \ref{Th1.1}, \ref{Th1.2} and Corollary \ref{Cor 3.1}, the assumption ``$B$ is convex with respect to $p$'' is assumed to guarantee the global second derivative estimate \eqref{global 2nd bound regular A}, see Theorem 3.1 in \cite{JT-oblique-II}. As in Remark 3.2 in \cite{JT-oblique-II}, for $k$-Hessian operators $\mathcal{F}_k=(S_k)^{1/k}$ in the cases $k=1, 2$ or $n$, estimate \eqref{global 2nd bound regular A} can hold without the convexity assumption on $B$ with respect to $p$. Consequently, in these particular $k$-Hessian cases, Theorems \ref{Th1.1}, \ref{Th1.2} and Corollary \ref{Cor 3.1} can still hold without the hypothesis that $B$ is convex with respect to $p$.
\end{remark}

\begin{remark}
Note that Theorem \ref{Th1.2} and Corollary \ref{Cor 3.1} embrace many examples of matrices $A$ and operators $\mathcal{F}$. In particular, one can refer to \cite{MTW2005, TruWang2009, Loeper2009,JT-oblique-I} for examples of the matrices $A$ and \cite{JT-oblique-I, Tru2019} for examples of the operators $\mathcal{F}$. Note that when the operator $\mathcal{F}$ is given by $\log \det$, (or $\det^{1/n}$), and $A$ is a strictly regular matrix generated by an optimal transportation cost function, then we need only assume $B = B(x)$ is uniformly H\"older continous  for global second derivative bounds \cite{HJL2015}, as in the linear Schauder theory and it would be interesting to know if such type of results extend more generally to say $k$-Hessians or just regular matrix functions $A$.
\end{remark}

\begin{remark}
As in \cite{JT-oblique-I}, the main examples of the admissible cones $\Gamma$ in Theorem \ref{Th1.1} and Theorem \ref{Th1.2} are the G{\aa}rding's cones $\Gamma_k$ and the $k$-convex cones $\mathcal{P}_k$, which are defined by \eqref{Garding's cone} and
\begin{equation}
\mathcal{P}_k :=\left\{ r\in \mathbb{S}^n| \ \sum\limits_{s=1}^k \lambda_{i_s}>0 \right\},
\end{equation}
where $i_1, \cdots, i_k \subset \{1, \cdots, n\}$, $\lambda(r)=(\lambda_1(r), \cdots, \lambda_n(r))$ denote the eigenvalues of the matrix $r\in \mathbb{S}^n$. 
Note that these two kinds of cones $\Gamma_k$ and $\mathcal{P}_k$ satisfy $\Gamma_n \subset \Gamma_k \subset \Gamma_1$ and $\Gamma_n \subset \mathcal{P}_k \subset \Gamma_1$ for $k=1, \cdots, n$.
For the background and inclusion relations of the cones $\Gamma_k$ and $\mathcal{P}_k$, one can refer, for example, to \cite{LCJ2017}; see also \cite{Tru2019} for more general families. 
\end{remark}



\end{document}